\documentclass[11pt, reqno]{amsart}
\newtheorem{theorem}{Theorem}[section]
\newtheorem{lemma}[theorem]{Lemma}
\newtheorem{prop}[theorem]{Proposition}
\newtheorem{cor}[theorem]{Corollary}
\newtheorem{example}[theorem]{Example}
\newtheorem{definition}[theorem]{Definition}

\theoremstyle{remark}
\newtheorem{remark}[theorem]{Remark}

\DeclareMathOperator{\id}{id}

\newcommand{\set}[1]{\left\{#1\right\}}
\newcommand{\parens}[1]{\left(#1\right)}

\renewcommand{\bar}[1]{\overline{#1}}

\newcommand{\ep}{\epsilon}
\newcommand{\la}{\lambda}

\newcommand{\N}{\mathbb{N}}
\newcommand{\Z}{\mathbb{Z}}

\newcommand{\Q}{\mathbb{Q}}

\newcommand{\C}{\mathbb{C}}

\newcommand{\cA}{\mathcal{A}}

\newcommand{\cU}{U}

\newcommand{\fsl}{\mathfrak{sl}}

\usepackage{ mathrsfs, amsmath, amsthm, amscd, amsfonts,
amssymb, graphicx, color, tikz, framed, lipsum}
\usetikzlibrary{matrix,arrows,patterns}
\numberwithin{equation}{section}

\newcommand{\f}{\mathbf{f}}

\newcommand{\barmap}{\bar{\phantom{c}}}

\newcommand{\cL}{\mathcal{L}}

\newcommand{\bbinom}[2]{\begin{bmatrix}#1 \\ #2\end{bmatrix}}

\newcommand{\AU}{\phantom{}_{\cA}\cU}
\newcommand{\CU}{\phantom{}^{\C}U_0}
\newcommand{\QU}{U}
\newcommand{\dotAU}{\phantom{}_{\cA}\dot{U}}

\newcommand{\dotU}{\dot{\cU}}

\newcommand{\piA}{\cA^\pi}
\newcommand{\piQ}{\Q(q)^\pi}
\newcommand{\piU}{\cU^{\pi}}
\newcommand{\dotpiU}{\dot{U}^{\pi}}

\newcommand{\osp}{\mathfrak{osp}}

\newcommand{\qfact}[1]{[#1]^!}

\newcommand{\parity}[1]{{p(#1)}}

\everymath=\expandafter{\the\everymath\displaystyle} 

\begin{document}
\title{\textsc{Canonical basis for quantum} $\osp(1|2)$}

\author[Sean Clark and Weiqiang Wang]{Sean Clark and Weiqiang Wang}
\address{Department of Mathematics, University of Virginia, Charlottesville, VA 22904}
\email{sic5ag@virginia.edu (Clark), \quad ww9c@virginia.edu (Wang)}


\begin{abstract}
We introduce a modified quantum enveloping algebra as well as a
(modified) covering quantum algebra for the ortho-symplectic Lie
superalgebra $\osp(1|2)$. Then we formulate and compute the
corresponding canonical bases, and relate them to the counterpart
for $\mathfrak{sl}(2)$. This provides a first example of canonical
basis for quantum superalgebras.
\end{abstract}


\maketitle


\section{Introduction}

The canonical basis of Lusztig \cite{Lu1} and Kashiwara \cite{Ka}
has served as an important motivation of the categorification of
quantum enveloping algebras. In a recent paper \cite{HW} of David
Hill and the second author, a class of (halves of) quantum Kac-Moody
{\em super}algebras has been categorified, and in addition, it was
suggested for the first time to use a novel bar-involution to
construct canonical basis of quantum Kac-Moody superalgebras and
their integrable modules. We refer the reader to {\em loc. cit.} for
extensive references in the fast-growing area of categorification.

The aim of this paper is to formulate and compute the canonical
bases for a modified quantum enveloping superalgebra $\dotU$ as well
as for a (modified) covering quantum superalgebra $\dotU^\pi$
associated to the ortho-symplectic Lie superalgebra $\osp(1|2)$.
Since canonical basis has never been formulated before for quantum
superalgebras, we find it desirable to work out the  formulas and
constructions in detail in this rank one setting. The new features
and connections observed in this paper will be instrumental in a
forthcoming work \cite{CHW} joint with David Hill on canonical basis
for general quantum Kac-Moody superalgebras.

The algebra $\dotU$ is modified from a quantum enveloping
superalgebra $U$ for $\osp(1|2)$ by adding idempotents, following
\cite{BLM, Lu}. Our (Hopf) superalgebra $U$ is defined as a direct
sum of $\Q(q)$-superalgebras $U_0$ and $U_1$, where $U_0$ and $U_1$
differ somewhat from the quantum $\osp(1|2)$ used in the literature
(cf. \cite{KR, AB, Zou, Jeo, B2}). In contrast to those variants,
our algebras $U_0, U_1$ and $U$ are well suited for introducing a
bar-involution and an integral form as needed in the construction of
canonical basis, and the modified algebra $\dotU$ has an intrinsic
description. The bar-involution on $U$ and $\dotU$ used in this
paper has the unusual feature that it sends a quantum parameter $q$
to $-q^{-1}$ (cf. \cite{HW}).

The complexified algebras $\CU$ for $U_0$ and ${}^\C U_1$ for $U_1$
are shown to be isomorphic, and finite-dimensional simple modules of
$\CU$ were classified in \cite{Zou} in terms of highest weights
labeled by pairs $(n,\pm)$ for $n\in\N$. We show those even-weight
(i.e., odd-dimensional) simple $\CU$-modules arise from the simple
$U_0$-modules while those odd-weight (i.e., even-dimensional) simple
$\CU$-modules arise from the simple $U_1$-modules.

Following \cite{HW}, we introduce a covering quantum algebra $U^\pi$
for $\osp(1|2)$ with an additional parameter $\pi$ such that
$\pi^2=1$. The covering algebra $U^\pi$ admits a modified version
$\dotU^\pi$ too. The structure constants when multiplying the
canonical basis elements in $\dotU^\pi$ are positive integer Laurent
polynomials in $q$ and $\pi$. We expect that the algebra $\dotU^\pi$
and its canonical basis can be categorified in a generalized
framework of spin nilHecke algebras (\`a la Lauda \cite{La} for
$\dot{U}_q(\mathfrak{sl}(2))$, where $\pi$ again is categorified as
a parity shift functor as in \cite{HW}. The  algebras $U^\pi$ and
$\dotU^\pi$ specialize when $\pi=1$ to $U_q(\fsl(2))$ and its
modified version, and specialize when $\pi=-1$ to $\dotU$ and
$\dotU^\pi$. In particular, the canonical basis for $\dotU^\pi$ are
shown to specialize when $\pi=1$ and $\pi=-1$ to the canonical basis
for modified quantum $\mathfrak{sl}(2)$ \cite{Lu} and for $\dotU$,
respectively. In other words, our constructions and formulas can be
regarded as a $\pi$-enhanced version of their counterparts for
quantum $\fsl(2)$.

It is well known that Lie superalgebra $\osp(1|2)$ admits only
odd-dimensional simple modules. In contrast, the quantum $\osp(1|2)$
as defined in this paper has richer representation theory, which are
compatible with the categorification construction and also with
quantum $\fsl(2)$. All these will afford a natural generalization in
the setting of quantum Kac-Moody superalgebras.

This paper is organized as follows. In
Section~\ref{sec:algebradefs}, we define the algebras $U_0, U_1$ and
study their basic structures including the integral forms and
(anti-)automorphisms. In Section~\ref{sec:modules}, we classify the
finite-dimensional simple weight modules of $U_0$ and $U_1$. In
Section~\ref{sec:algU}, we show $U=U_0\oplus U_1$ has a natural Hopf
superalgebra structure. In Section~\ref{sec:CBRmatrix}, we find an
explicit formula for the quasi-$R$-matrix of $U$, which is then used
in defining the bar-involution for a tensor product of modules. The
canonical basis on the tensor product of two finite-dimensional
$U$-modules is computed. In Section~\ref{sec:CBdotU}, we define the
modified algebra $\dotU$, compute its canonical basis, and formulate
a bilinear form on $\dotU$. In Section~\ref{sec:covering}, we
formulate in the framework of covering algebras variants of
constructions and results in the previous sections.

\vspace{.3cm}

 {\bf Acknowledgments.}
We thank David Hill for fruitful collaboration and stimulating
discussions on closely related projects. The second author is
partially supported by an NSF grant DMS--1101268.

\section{Structures of quantum $\osp(1|2)$}
\label{sec:algebradefs}

\subsection{Algebra $U_0$}

Set
$$
\pi=-1
$$
throughout this paper except the final Section~\ref{sec:covering},
and we will use the symbol $\pi$ for the super signs in
superalgebras arising from exchanges of odd elements. This allows us
to state clean commutation formulas, and to recover many classical
formulas for quantum $\mathfrak{sl}(2)$ by simply dropping $\pi$.

\begin{definition}  \label{def:Uq}
The algebra  $U_0$ is the
$\Q(q)$-algebra generated by $E, F, K$, and $K^{-1}$, subject to the
relations:
\begin{enumerate}
 \item $KK^{-1}=1=K^{-1}K$;
 \item $KEK^{-1}=q^2E$, \quad $KFK^{-1}=q^{-2}F$;
 \item $EF-\pi FE = \frac{K-K^{-1}}{\pi q - q^{-1}}$.
\end{enumerate}
\end{definition}

\begin{remark}
There has been definitions for quantum enveloping algebra of $\osp(1|2)$, which differ
from $U_0$ by a different rescaling of the
relation~(3) above. A version of $U_q(\osp(1|2))$ appeared in
\cite{AB, Zou}, where (3) is replaced by
\begin{enumerate}
 \item[(3a)]
$EF-\pi FE = \frac{K-K^{-1}}{q^2-q^{-2}}.$
\end{enumerate}
On the other hand, the definition used in \cite{Jeo} replaces (3) by
\begin{enumerate}
 \item[(3b)]
$EF-\pi FE =\frac{K-K^{-1}}{q-q^{-1}}.$
\end{enumerate}

These variants of $U_q(\osp(1|2))$ are all isomorphic to $U_0$ as
$\Q(q)$-algebras, with isomorphisms given by fixing $F$ and $K$, and
then by rescaling $E$ by suitable scalars in $\Q(q)$. Our
Definition~\ref{def:Uq} is most suitable for introducing an integral
form $\AU$ and a bar-involution $\bar{\phantom{c}} \,: U \rightarrow
U$ below. As we shall see, (3b) is not bar-invariant under the
bar-involution \eqref{eq:bar}, while (3a) is not well suited for
constructing an integral form.
\end{remark}

\subsection{Algebra $U_1$}

We introduce a variant of quantum enveloping algebra for
$\osp(1|2)$.

\begin{definition}  \label{def:Uq2}
The  algebra  $U_1$ is the $\Q(q)$-algebra generated by $E, F, K$,
and ${K}^{-1}$, subject to the relations:
\begin{enumerate}
 \item $K{K}^{-1}=1={K}^{-1}K$;

 \item $KE{K}^{-1}=q^2E$, \quad $KF{K}^{-1}=q^{-2}F$;

 \item $EF-\pi FE = \frac{\pi K-{K}^{-1}}{\pi q - q^{-1}}$.
\end{enumerate}
\end{definition}

Note the difference between definitions of $U_0$ and $U_1$ lies in the
relation (3).

\begin{remark}\label{rem:epsrels}
As we need to mix the use of $U_0$ and $U_1$, we shall denote the
generators for $U_0$ (respectively, $U_1$) by $E_0,F_0,K_0$
(respectively, $E_1, F_1, K_1$). Then the defining relations of
$U_\epsilon$ ($\epsilon=0,1$) can be succinctly rewritten as
\begin{enumerate}
\item $K_\epsilon  {K_\epsilon}^{-1}=1={K_\epsilon}^{-1}K_\epsilon$;

\item
$K_\epsilon E_\epsilon {K_\epsilon}^{-1}=q^2E_\epsilon$, \quad
$K_\epsilon F_\epsilon{K_\epsilon}^{-1}=q^{-2}F_\epsilon$;

\item
$E_\epsilon F_\epsilon-\pi F_\epsilon E_\epsilon =
\frac{\pi^\epsilon K_\epsilon-{K_\epsilon}^{-1}}{\pi q - q^{-1}}$.
\end{enumerate}
The algebra  $U_\ep$ is naturally a superalgebra by letting $E_\ep,
F_\ep$ be odd and $K_\ep^{\pm 1}$ be even.

\end{remark}

\subsection{Complexification}

Fix a square root $\sqrt{\pi} \in \C$. For $\epsilon=0,1$, denote
$$
{}^\C U_\ep=\C(q)\otimes_{\Q(q)}U_\ep.
$$
Though $U_0$ and $U_1$ are not isomorphic as $\Q(q)$-algebras, we have
the following.

\begin{lemma} \label{lem:UU}
There is an isomorphism of $\C(q)$-algebras $\flat: {}^\C U_1
\rightarrow {}^\C U_0$ such that
\begin{align*}
\flat(F_1)=F_0, & \quad \flat(E_1)=\sqrt{\pi}E_0,\quad
\flat(K_1)=\sqrt{\pi}^{\, -1}K_0.
\end{align*}
\end{lemma}
We may formally regard $U_0$ and $U_1$ as two different real forms
for the same $\C(q)$-algebra. They share many of the same structural
properties, and the proofs of these properties are quite similar.
The rationale of introducing $U_1$ besides $U_0$ comes from
Sections~\ref{sec:modules} and \ref{sec:CBdotU}.

\subsection{PBW and gradings}
\label{sec:gradings}

Clearly the elements $F_\epsilon^{a} K_\epsilon^b E_\epsilon^{c} $
with $a,c\in \N$ and $b\in \Z$ span $U_\epsilon$ since any monomial
in $E_\epsilon$, $F_\epsilon$, and $K_\epsilon$ can be expressed as
a sum of such elements by using the defining relations. Proving
linear independence can be done as in \cite[1.5]{Jan}. Hence we
obtain the following.

\begin{prop} \label{prop:PBW}
The algebra $U_\epsilon$, for $\epsilon=0,1$,  has the following (PBW) bases:
$$
\set{F_\epsilon^{a} K_\epsilon^b E_\epsilon^{c} | a,c\in \N, b\in \Z}, \qquad
\set{E_\epsilon^{a} K_\epsilon^b F_\epsilon^{c} | a,c\in \N, b\in \Z}.
$$
\end{prop}

Let $U_\epsilon^+$ be the subalgebra of $U_\epsilon$ generated by
$E_\epsilon$, $U_\epsilon^-$ be the subalgebra generated by
$F_\epsilon$, and $U_\epsilon^0$ be the subalgebra generated by
$K_\epsilon, K_\epsilon^{-1}$.

The algebra $U_\epsilon$ has two natural gradings on it: the $\Z$-grading
arising from weight space decomposition  of $\osp(1|2)$, and a
parity $\Z_2$-grading arising from the superalgebra structure of
$\osp(1|2)$. The {\em parity $\Z_2$-grading} on the algebra $U_\epsilon$ is
defined by
\[
p(E_\epsilon)=p(F_\epsilon)=1, \qquad p(K_\epsilon)=p(K_\epsilon^{-1})=0.
\]
The {\em weight $\Z$-grading} on the algebra $U_\ep$ (which is the same
as a weight space decomposition in our rank one setting) is defined
by
\[
|E_\epsilon|=2,\qquad |F_\epsilon|=-2,\qquad |K_\epsilon|=|K^{-1}_\epsilon|=0,
\]
since the defining relations are clearly homogeneous with respect to
this definition. We have
\[
U_\epsilon=\bigoplus_{i\in 2\Z} U_\epsilon(i),
\qquad U_\epsilon(i)=\set{u\in U_\epsilon| K_\epsilon uK_\epsilon^{-1}=q^iu}.
\]

\subsection{The $\cA$-subalgebra}

Let
$$
\cA=\Z[q,q^{-1}], \qquad \N =\{0,1,2,\ldots\}.
$$
For $n\in\Z$ and $a\in\N$, we define the {\em super quantum integer}
or {\em $(q,\pi)$-integer}
\begin{equation}  \label{eq:qn}
 [n]=\frac{(\pi q)^n-q^{-n}}{\pi q - q^{-1}},
\end{equation}
and then define the corresponding factorials and binomial coefficients
\begin{equation}  \label{eq:qbinom}
\qfact{a}=\prod_{i=1}^a [i],\qquad
\bbinom{n}{a}=\frac{\prod_{i=1}^a[n+i-a]}{\qfact{a}}.
\end{equation}
We adopt the convention that $\qfact{0}=1$. Note that
$\bbinom{n}{a}=\frac{\qfact{n}}{\qfact{a}\qfact{n-a}}$, for  $n\geq
a\geq 0$. One checks that $[n]\in\cA, \bbinom{n}{a} \in\cA$. A
straightforward computation gives us
\begin{align}  \label{eq:n-n}
[-n]=-\pi^n [n],
 \qquad
\bbinom{n}{a} =(-1)^a \pi^{na + \binom{a}{2}} \bbinom{a-n-1}{a}.
\end{align}

We use these super quantum integers to define the divided powers:
\begin{equation} \label{eq:dpower}
E_\epsilon^{(a)}=\frac{E_\epsilon^a}{\qfact{a}}, \qquad
F_\epsilon^{(a)}=\frac{F_\epsilon^a}{\qfact{a}}.
\end{equation}
It is understood that $E_\epsilon^{(0)}=F_\epsilon^{(0)}=1$. For
$n\in\Z, a\in\N$, we also define the following elements in $U_\ep$
(compare \cite{Jan}):
\begin{equation} \label{eq:Kn}
[K_\epsilon;n]=\frac{(\pi q)^n \pi^\epsilon
K_\epsilon-q^{-n}K_\epsilon^{-1}}{\pi q - q^{-1}}, \qquad
\bbinom{K_\epsilon;n}{a}=\frac{\prod_{j=1}^a
[K_\epsilon;n+j-a]}{\qfact{a}}.
\end{equation}
We let $\AU_\epsilon$ be the $\cA$-subalgebra of $U_\epsilon$
generated by $E_\epsilon^{(a)}, F_\epsilon^{(a)}, K_\epsilon^{\pm
1}, \bbinom{K_\epsilon;n}{a}$, for $n\in \Z, a\in\N$. 
\subsection{Automorphisms}\label{Automorphisms}

Following a key observation in \cite{HW}, we define the
$\Q$-automorphism of $\Q(q)$, denoted by $\bar{\phantom{c}}$, such that
\begin{equation} \label{eq:bar}
\bar{q}=\pi q^{-1}.
\end{equation}
Note that the super quantum integers are bar-invariant. A map $\phi$
from a $\Q(q)$-algebra $A$ to itself is called {\em antilinear} if
$\phi(g(q)a)=\bar{g(q)}\phi(a)$, for $g(q) \in\Q(q)$. We also adopt
the convention that an {\em anti-homomorphism} $f$ on $A$ is a
$\Q(q)$-linear map satisfying $f(xy)=f(y)f(x)$, for $x,y\in A$.
Below we shall denote by $D_4$ the dihedral group of order $8$.

\begin{prop}  \label{prop:autom}
Let $\ep\in \{0,1\}.$
\begin{enumerate}
\item
There is a $\Q(q)$-antilinear involution
$\psi_\epsilon:U_\epsilon\rightarrow U_\epsilon$ such that
\[
\psi_\epsilon(E_\epsilon)=E_\epsilon, \qquad
\psi_\epsilon(F_\epsilon)=F_\epsilon, \qquad
\psi_\epsilon(K_\epsilon)=\pi^\epsilon K_\epsilon^{-1};
 \]
($\psi_\epsilon$ is referred to as the {\em bar involution} and also
denoted by $\bar{\phantom{c}} \,: U_\epsilon \rightarrow U_\epsilon$).

\item
There is a $\Q(q)$-linear automorphism
$\omega_\epsilon:U_\epsilon\rightarrow U_\epsilon$ such that
\[
\omega_\epsilon(E_\epsilon)=F_\epsilon, \qquad
\omega_\epsilon(F_\epsilon)=\pi^{1-\epsilon} E_\epsilon, \qquad
\omega_\epsilon(K_\epsilon)= K_\epsilon^{-1};\]

\item
There is a $\Q(q)$-linear anti-involution
$\tau_\epsilon:U_\epsilon\rightarrow U_\epsilon$ such that
\[
\tau_\epsilon(E_\epsilon)=\pi^{1-\epsilon} E_\epsilon, \qquad
\tau_\epsilon(F_\epsilon)=F_\epsilon, \qquad
\tau_\epsilon(K_\epsilon)= K_\epsilon^{-1};
 \]

\item
There is a $\Q(q)$-linear anti-involution
$\rho_\epsilon:\QU_\epsilon\rightarrow \QU_\epsilon$ such that
 \[
\rho_\epsilon(E_\epsilon)=qK_\epsilon F_\epsilon, \qquad
\rho_\epsilon(F_\epsilon)=qK_\epsilon^{-1}E_\epsilon, \qquad
\rho_\epsilon(K_\epsilon)=K_\epsilon.
 \]

\item
The subgroup of (anti-)automorphisms on $U_\epsilon$ generated by
$\omega_\epsilon, \tau_\epsilon, \psi_\epsilon$ is isomorphic to
$D_4\times \Z_2$ for $\ep=0$ and to $\Z_2\times\Z_2\times\Z_2$ for
$\ep=1$. More precisely,
\begin{align*}
\omega_0^4=1, \;\omega_1^2=1, \quad &
\tau_0\omega_0=\omega_0^3\tau_0, \quad
\tau_1\omega_1=\omega_1\tau_1,
 \\
\tau_\epsilon^2=\psi_\epsilon^2=1, \quad & \psi_\epsilon
\tau_\epsilon=\tau_\epsilon \psi_\epsilon, \quad
\psi_\epsilon\omega_\epsilon=\omega_\epsilon\psi_\epsilon.
\end{align*}
\end{enumerate}
\end{prop}

\begin{proof}
This is proved by a direct computation, and let us suppress the
subscript $\ep$.  To illustrate, let us verify that the (most
involved) commutation relation (3)  in Remark~\ref{rem:epsrels}
between $E$ and $F$ is preserved under these maps. Since $\psi$
fixes $E$, $F$, and $\pi^\epsilon K - K^{-1}$, it preserves the
relation between $E$ and $F$, whence (1).

To verify for (2), we compute
\begin{align*}
\omega(EF-\pi FE)&=\pi^{1-\epsilon}FE-\pi^\epsilon EF =
-\pi^{\epsilon}(EF-\pi FE),
 \\
\omega\parens{\frac{\pi^\epsilon K-K^{-1}}{\pi q - q^{-1}}}&=
-\pi^{\epsilon}\parens{\frac{\pi^\epsilon K-K^{-1}}{\pi q -
q^{-1}}}.
\end{align*}

For (4), we further compute
\begin{align*}
\rho(EF-\pi FE)&=q^2K^{-1}EKF-\pi q^2KFK^{-1}E = EF-\pi FE,
 \\
\rho\parens{\frac{\pi^\epsilon K-K^{-1}}{\pi q - q^{-1}}} &=
\frac{\pi^\epsilon K-K^{-1}}{\pi q - q^{-1}}.
\end{align*}
The calculation for $\tau$ in (3) is exactly the same as for
$\omega$. Finally (5) may be quickly verified by checking on the
generators.
\end{proof}

Let $\ep\in \{0,1\}$. By Proposition~\ref{prop:autom}, we have the
following identities in $U_\ep$: for $n\in \Z, a\in\N$,
\begin{align} \label{eq:omega}
\begin{split}
\omega_\epsilon(E_\ep^{(r)})=F_\ep^{(r)}, &\qquad
 \omega_\epsilon(F_\ep^{(r)})=\pi^{r(1-\ep)} E_\ep^{(r)},
 \\
\omega_\epsilon([K_\epsilon;n])&=-\pi^{\ep+n} [K_\epsilon;-n],
 \\
\omega_\epsilon\parens{\bbinom{K_\epsilon;n}{a}} &=(-1)^a\pi^{\ep
a+na-\binom{a}{2}} \bbinom{K_\epsilon;a-n-1}{a}.
\end{split}
\end{align}
It is straightforward to check the following identities in $\AU_\ep$:
for $a,b, c, s \in \Z$,
\begin{align}
[b+c][K_\epsilon;a] &=[b][K_\epsilon;a+c]+\pi^b [c][K_\epsilon;a-b],
  \notag   \\
E_\epsilon[K_\epsilon;s] &= [K_\epsilon;s-2]E_\epsilon,
  \label{eq:abc} \\
F_\epsilon[K_\epsilon;s] &=[K_\epsilon;s+2]F_\epsilon. \notag
\end{align}

\subsection{Commutation relations}

\begin{lemma}\label{divpowcom}
Let $\ep\in \{0,1\}$.
The following identities hold in $\AU_\epsilon$: for $r, s \geq 1$,
\begin{enumerate}
\item
$\pi^s E_\epsilon F_\epsilon^{(s)}=F_\epsilon^{(s)}E_\epsilon+\pi
F_\epsilon^{(s-1)}[K_\epsilon;1-s]$;

\item
$\pi^{rs} E_\epsilon^{(r)}F_\epsilon^{(s)}=\sum_{i=0}^{\min(r,s)}
\pi^{\binom{i+1}{2}}
F_\epsilon^{(s-i)}\bbinom{K_\epsilon;2i-(r+s)}{i}E_\epsilon^{(r-i)}$;

\item
$\pi^s F_\epsilon E_\epsilon^{(s)}= E_\epsilon^{(s)} F_\epsilon-
\pi^{1-s} E_\epsilon^{(s-1)}[K_\epsilon;s-1]$;

\item
$\pi^{rs} F_\epsilon^{(s)}E_\epsilon^{(r)}=\sum_{i=0}^{\min(r,s)}
(-1)^i\pi^{i(r+s)}
E_\epsilon^{(r-i)}\bbinom{K_\epsilon;r+s-(i+1)}{i}F_\epsilon^{(s-i)}$.
\end{enumerate}
\end{lemma}

\begin{proof}
The first two identities (1) and (2) can be proven using induction.
Fix $\epsilon\in\set{0,1}$. Again, we suppress the subscripts
throughout the proof.

(1). The base case $s=1$ is a defining relation for $U_\ep$. Now
suppose that the identity (1) holds for some $s$. Then
\begin{align*}
\pi^{s+1} E F^{(s)}F&=\pi F^{(s)}E F+\pi^2 F^{(s-1)}[K;1-s]F
 \\
&=F^{(s)}FE+\pi F^{(s)}[K;0]+\pi^2[s] F^{(s)}[K;-1-s]
 \\
&=F^{(s)}FE+\pi F^{(s)} ([K;0]+\pi [s] [K;-1-s])
 \\
&=F^{(s)}FE+\pi F^{(s)}[s+1][K;-s]
\end{align*}
The last equality follows from \eqref{eq:abc} with $a=-s$, $b=1$,
and $c=s$. Dividing both sides by $[s+1]$ finishes the induction
step.

(2). We proceed by induction on $r$, with the case case for $r=1$
being (1). Suppose now that the identity (2) holds for some $r$.
Then
\begin{align}
\pi^{rs+s} EE^{(r)}F^{(s)}&=\sum_{i=0}^{\min(r,s)} \pi^{\binom{i+1}{2}}
\pi^s E F^{(s-i)}\bbinom{K;2i-(r+s)}{i}E^{(r-i)}
\notag \\
&=\sum_{i=0}^{\min(r,s)} \pi^{\binom{i+1}{2}}
\pi^i F^{(s-i)}E \bbinom{K;2i-(r+s)}{i}E^{(r-i)}
\notag \\
&\ \ +\sum_{i=0}^{\min(r,s)}\pi^{\binom{i+1}{2}}\pi^{i+1}
F^{(s-i-1)}[K;1+i-s] \bbinom{K;2i-(r+s)}{i}E^{(r-i)}
\notag \\
&=\sum_{i=0}^{\min(r,s)} \pi^{\binom{i+1}{2}}
\pi^i F^{(s-i)} \bbinom{K;2i-(r+s+2)}{i}EE^{(r-i)}
\notag \\
&\ \ +\sum_{i=1}^{\min(r,s)+1}\pi^{\binom{i}{2}}\pi^{i}
F^{(s-i)}[K;i-s] \bbinom{K;2i-(r+s+2)}{i-1}E^{(r-i+1)}
\notag \\
&=\sum_{i=0}^{\min(r+1,s)}
\pi^{\binom{i+1}{2}}F^{(s-i)}X_{i}E^{(r+1-i)}. \label{eq:ErFs}
\end{align}
Here $X_0=[r+1]$, $X_{r+1}=[r+1]\bbinom{K;r+1-s}{r+1}$ if $r< s$,
and for $1\leq i \leq \min(r,s)$,
\begin{align*}
X_i&=  \pi^{i}[r+1-i]\bbinom{K;2i-(r+s+2)}{i}+[K;i-s] \bbinom{K;2i-(r+s+2)}{i-1}\\
&=[i]^{-1}\bbinom{K;2i-(r+s+2)}{i-1} \big(\pi^i [r+1-i]
[K;i-(r+s+1)]+[i][K;i-s]\big)
 \\
&\stackrel{(*)}{=}[i]^{-1}\bbinom{K;2i-(r+s+2)}{i-1} [r+1] [K;2i-(r+s+1)]\\
&=[r+1]\bbinom{K;2i-(r+s+1)}{i}.
\end{align*}
The equality ($*$) above follows from \eqref{eq:abc} with
$a=2i-(r+s+1)$, $b=i$, and $c=r+1-i$. Dividing both sides of
\eqref{eq:ErFs} by $[r+1]$ we obtain (2).

The identities (3) and (4) follow by applying the automorphism
$\omega_\epsilon$ to (1) and (2) and using \eqref{eq:omega}.
\end{proof}

\section{Finite-dimensional representations} 
\label{sec:modules}


\subsection{Weight $U_\epsilon$-modules}\label{subsec:Uepmodules}
Let us now turn to $U_\ep$-modules, for $\ep=1,2$. We will call a
$U_\epsilon$-module $M$ a \textit{weight module} if the action of
$K$ on $M$ is semisimple with finite-dimensional eigenspaces (i.e.,
weight spaces). The {\em Verma module} of $U_\ep$ of highest weight
$\la\in \Q(q)$ is defined to be
$$
M_\ep^\la =U_\ep/ (U_\ep E_\ep +U_\ep (K_\ep -\la)),
$$
with an even highest weight vector denoted by $\nu$. Then by
Proposition~\ref{prop:PBW} $M_\ep^\la$ has a basis given by
$F_\ep^{(k)} \nu$, for $k\ge 0$. Denote by $L_\ep^\la$ for now the
unique irreducible quotient module of $M_\ep^\la$. We observe the
following three statements are equivalent: (i) The $U_\ep$-module
$M_\ep^\la$ is reducible; (2) $M_\ep^\la$ admits a (singular) vector
$F^{(t)} \nu$ for some $t>0$ annihilate by $E_\ep$; (3) $L_\ep^\la$
is finite dimensional. By Lemma~\ref{divpowcom},  we have
\[
E_\epsilon F_\epsilon^{(t)} \nu =\pi F_\epsilon^{(t-1)}[K_\epsilon;1-t]\nu.
\]
A quick calculation using this equation to locate a possible
singular vector in $M_\ep^\la$ leads to the following.

\begin{prop}  \label{prop:Umod}
Let  $\ep\in \{0,1\}$.
\begin{enumerate}
\item
$M_\ep^\la$ is an irreducible $U_\ep$-module, unless $\la =\pm q^n$
for $n\in \ep+2\N$.

\item
For each $n\in \ep+2\N$, there is a unique pair of
$(n+1)$-dimensional simple $U_\epsilon$-modules
$L(n,\pm):=L_\ep^{\pm q^n}$ of highest weight $\pm q^{n}$. Moreover,
any finite-dimensional simple weight $U_\epsilon$-module is
isomorphic to one such module.
\end{enumerate}
\end{prop}

This result should be compared to the classification of
finite-dimensional simple modules for $\CU$ below.

\begin{prop} \cite{Zou}   \label{prop:zou}
For each $n\in \N$, there are two non-isomorphic $(n+1)$-dimensional
$\CU$-modules over $\C(q)$ of highest weight $\pi^{n^2/2}q^{n}$.
Moreover, any finite-dimensional $\CU$-module is completely
reducible.
\end{prop}

\begin{remark}
Note that the weights of the simple $\CU$-modules for $n$ odd in
Proposition~\ref{prop:zou} involve complex number $\sqrt{\pi}$, and
so they cannot be realized as $U_0$-modules over $\Q(q)$. This
partially motivated our introduction of $U_1$.
\end{remark}

\begin{remark}
Proposition~\ref{prop:zou} remains to be valid if we classify
finite-dimensional modules of $Q[\sqrt{\pi}](q)\otimes_{\Q(q)}U_0$
over the field $\Q[\sqrt{\pi}](q)$ instead of $\C(q)$.

Note that the ``weight" $U_\epsilon$-module condition in
Proposition~\ref{prop:Umod} is necessary over $\Q(q)$. Indeed, if we
view the $Q[\sqrt{\pi}](q)$-vector space underlying a 2-dimensional
module of $Q[\sqrt{\pi}](q)\otimes_{\Q(q)}U_0$ as a $\Q(q)$-vector
space, we obtain a 4-dimensional $U_0$-module which is not a weight
module.
\end{remark}

\subsection{Complete reducibility}

It has been known that there is a Casimir element for (a version of)
the algebra $U_0$ (see e.g. \cite{AB}). Let $\ep\in \{0,1\}$. We
adapt this construction to  the algebras $U_\ep$. We will proceed as
in \cite[\S\S 2.7-2.9]{Jan}. Set

\begin{equation}\label{eq:qcasimir}
C_\epsilon=\pi F_\epsilon E_\epsilon + \frac{ \pi^{1-\epsilon}
K_\epsilon q + K_\epsilon^{-1} q^{-1}}{ (\pi q - q^{-1})^2 }.
\end{equation}
One rewrites using defining relations of $U_\ep$ that
\[
C_\epsilon=E_\epsilon F_\epsilon + \frac{\pi^\epsilon K_\epsilon
q^{-1} + \pi K_\epsilon^{-1} q}{ (\pi q - q^{-1})^2 }.
\]
We note that $\omega_\ep(C_\ep)=\tau_\ep(C_\ep)=\pi^\ep C_\ep$.
Also, we have that
\begin{equation}  \label{eq:scent}
C_\epsilon E_\epsilon =\pi E_\epsilon C_\epsilon, \quad
C_\epsilon F_\epsilon =\pi F_\epsilon C_\epsilon, \quad
C_\epsilon K_\epsilon =K_\epsilon C_\epsilon.
\end{equation}
Indeed, clearly we have $
C_\epsilon K_\epsilon =K_\epsilon C_\epsilon.$ We compute
\begin{align*}
C_\epsilon E_\epsilon &=E_\epsilon F_\epsilon E_\epsilon
+ \frac{\pi^\epsilon K_\epsilon q^{-1}
+ \pi K_\epsilon^{-1} q}{ (\pi q - q^{-1})^2 }E_\epsilon
 \\
&=\pi\parens{ \pi E_\epsilon F_\epsilon E_\epsilon +
E_\epsilon\frac{\pi \pi^\epsilon K_\epsilon q + K_\epsilon^{-1}
q^{-1}}{ (\pi q - q^{-1})^2 }}=\pi E_\epsilon C_\epsilon.
\end{align*}
The remaining identity in \eqref{eq:scent} can be checked similarly.
It follows by \eqref{eq:scent} that $C_\ep^2$ is in the center of
$U_\ep$.

\begin{prop}  \label{prop:ssCat}
Let $\ep\in \{0,1\}$ and $n\in\Z$. Then,
\begin{enumerate}
\item
$C_\ep^2$ acts on the Verma module $M^{\pm q^n}_\ep$ as scalar
multiplication by $\frac{[n+1]^2}{(\pi q - q^{-1})^2}$.

\item
$C_\ep^2$ acts on $M^{\pm q^n}_\ep$ and $M^{\pm q^m}_\ep$  by the
same scalar if and only if $n=m$ or $n=-m-2\in \Z$; in particular,
$C_\ep^2$ acts as a different scalar on different pairs $L(n,\pm)$,
for $n\in\ep+2\Z_+$.

\item Any finite-dimensional weight $U_\ep$-module is completely reducible.
\end{enumerate}
\end{prop}

\begin{proof}
Let $\nu$ be the highest weight vector of $\Lambda_n$.
Using (\ref{eq:qcasimir}), we see that
$C_\ep^2\nu=\frac{[n+1]^2}{(\pi q - q^{-1})^2}\nu$. Since any $m\in
M^{\pm q^n}_\ep$ can be represented as $m=u\nu$ for $u\in U_\ep$,
$C_\ep^2m=C_\ep^2u\nu=uC_\ep^2\nu =(\pi q -
q^{-1})^{-2}{[n+1]^2}  m$, whence (1).

Now $[n+1]^2=[m+1]^2$ if and only if $(\pi q)^{n+1} - q^{-n-1}=\pm
((\pi q)^{m+1} - q^{-m-1})$, whence (2). For a given $n\in \Z_+$, by
weight considerations there is no nontrivial extension between
$L(n,+)$ and $L(n,-)$. We can prove (3) as is done in \cite[\S
2.9]{Jan}; that is, pick a composition series for $M$ and use a weight dimension argument
to show that composition factors are direct summands.
\end{proof}

\section{The Hopf superalgebra $U$}
 \label{sec:algU}
\subsection{Algebra $U$}
By the similarities of $U_\epsilon$ and $U_q(\fsl(2))$, we hope to
make sense that the tensor product of two odd-weight modules should
decompose as a sum of even-weight modules. It is therefore
convenient to combine $U_0$ and $U_1$ into a single algebra.

\begin{definition}  \label{def:Usplit}
The  algebra $U$ is defined to be the direct sum of algebras
$\cU=U_0\oplus U_1$, whose multiplication is denoted by $m$. Let
$e_0=(1,0)$ and $e_1=(0,1)$ be the central idempotents of $U$ with
$U_0=e_0\cU$, $U_1=e_1 \cU$ and $e_0e_1=0$; hence $\cU$ is a unital
algebra with $1=e_0+e_1$.
\end{definition}
Another possible way is to define a smaller single algebra so that both $U_0$ and $U_1$ become the quotient algebras, but we will not follow that route in this paper.

It is immediate that the direct sums (over $\ep=0,1$) of the
(anti-)automorphisms $\psi_\epsilon$, $\omega_\epsilon$,
$\tau_\epsilon$, and $\rho_\epsilon$ define (anti-)automorphisms
$\psi$, $\omega$, $\tau$, and $\rho$ on $\cU$, respectively. We also
have the $\cA$-subalgebra ${}_\cA U=\vphantom{|}_\cA U_0\oplus\vphantom{|}_\cA
U_1$ and a $\Z$-grading $U =\oplus_{i\in2\Z} U(i)$, where $\cU(i)=
U_0(i)+U_1(i)$. Since $\cU$ is a direct sum of unital algebras, each
$\cU$-module $M$ decomposes as $M=M_0\oplus M_1$ where $M_\ep=e_\ep
M$ is a $U_\ep$-module ($\ep=0,1$), and $U_1M_0=U_0M_1=0$. We shall
call a $\cU$-module $M=M_0\oplus M_1$ a weight module if $M_0$ and
$M_1$ are weight modules. We may restate Proposition~\ref{prop:Umod}
and Proposition~\ref{prop:ssCat}(3) in a form more commensurate with
Proposition~\ref{prop:zou} and also with representation theory of
$U_q(\fsl(2))$ (\cite{Jan}).

\begin{prop}\label{prop:cUmod}
For each $n\in \N$, there is a pair of non-isomorphic
$(n+1)$-dimensional simple $\cU$-modules denoted by $L(n,\pm)$ of
highest weight $\pm q^{n}$. Any finite dimensional simple weight
$\cU$-module is isomorphic to one such module. Moreover, any
finite-dimensional weight $U$-module is completely reducible.
\end{prop}
We will from now on concentrate only on $L(n):=L(n,+)$, since the
cases of $L(n,-)$ is completely parallel.

\subsection{Algebra $\f$}
\label{subsec:f}

Following Lusztig (\cite{Lu}), there is a free $\Q(q)$-algebra
$\f=\Q(q)[\theta]$, where $\theta$ has $\Z$-grading $2$ and parity
$p(\theta)=1$. We have natural $\Q(q)$-algebra isomorphisms
$(\cdot)_\epsilon^\pm: \f\rightarrow U_{\epsilon}^{\pm}$ given by
$\theta\mapsto \theta_\epsilon^+=E_\epsilon$ and $\theta\mapsto
\theta_\epsilon^-=F_\epsilon$. We define the maps
$(\cdot)^\pm:\f\rightarrow \cU$ by $u^\pm={u}_0^\pm\oplus
{u}_1^\pm$; that is, it is the diagonal embedding $\theta^+=E_0+E_1$
and $\theta^-=F_0+F_1$.

We can define a bilinear form on $\f$ such that
\begin{align}
(\theta, \theta) &= (1-\pi q^{-2})^{-1},
  \label{eq:bform1} \\
(\theta^{(a)}, \theta^{(b)}) &=\delta_{a,b}\prod_{s=1}^a
\frac{\pi^{s-1}}{1-(\pi q^{-2})^s} =\delta_{a,b}\pi^a q^{\binom{a+1}{2}}(\pi
q-q^{-1})^{-a} ([a]^!)^{-1}.
\label{eq:bform2}
\end{align}
A version of this bilinear form was first introduced in \cite{HW}
for quantum Kac-Moody superalgebras including $\mathfrak{osp}(1|2)$,
with a switch of $q$ with $q^{-1}$ in \eqref{eq:bform1}.

\subsection{The coproduct}

We endow the tensor product of superalgebras with the twisted multiplication
\[
(a\otimes b) * (c\otimes d) = \pi^{p(b)p(c)} ac\otimes bd.
\]

It is known that $U_0$ is a Hopf superalgebra (cf. \cite{Zou}). The
following lemma can be regarded as an extension of the coproduct on
$U_0$ (compare \cite[3.1.3]{Lu}).

\begin{lemma} \label{lem:coprods}
For fixed $\epsilon, \kappa \in\set{0,1}$, there is a unique
(super)algebra homomorphism
$\Delta_{\epsilon,\kappa}:U_{\epsilon+\kappa}\rightarrow
U_{\epsilon} {\otimes} U_{\kappa}$ satisfying
\begin{align*}
\Delta_{\epsilon, \kappa}(E_{\epsilon +
\kappa})&=E_{\epsilon}\otimes e_\kappa + \pi^{\epsilon}
K_{\epsilon}\otimes E_{\kappa},
 \\
\Delta_{{\epsilon},{\kappa}}(F_{\epsilon +
\kappa})&=F_{\epsilon}\otimes K_{\kappa}^{-1}+ e_\epsilon \otimes
F_{\kappa},
 \\
\Delta_{{\epsilon},{\kappa}}(K_{\epsilon +
\kappa})&=K_{\epsilon}\otimes K_{\kappa}.
\end{align*}
\end{lemma}
\begin{proof}
In the following, we shall suppress the subscripts on elements of
$U_{\epsilon+\kappa}$ since they are clear from context. We need to
prove that the defining relations of $U_{\epsilon+\kappa}$ are
preserved by $\Delta_{\epsilon,\kappa}$. We will only check the most
involved case as follows:
\begin{align*}
\Delta_{{\epsilon},{\kappa}}(E_{})\Delta_{{\epsilon},{\kappa}}(F_{})
&-\pi \Delta_{{\epsilon},{\kappa}}(F_{})\Delta_{{\epsilon},{\kappa}}(E_{})
 \\
&= [E_{\epsilon},F_{\epsilon}]\otimes
K_{\kappa}^{-1}+\pi^{{\epsilon}}K_{\epsilon}\otimes
[E_{\kappa},F_{\kappa}]
 \\
&= \frac{ (\pi^{\epsilon} K_{\epsilon}-K_{\epsilon}^{-1})\otimes
K_{\kappa}^{-1}+\pi^{{\epsilon}}K_{\epsilon}\otimes(\pi^{\kappa}
K_{\kappa} - K_{\kappa}^{-1})}{\pi q - q^{-1}}
 \\
&=\frac{\pi^{{\epsilon}+{\kappa}}K_{\epsilon}\otimes K_\kappa -
K_{\epsilon}^{-1}\otimes K_{\kappa}^{-1}}{\pi q - q^{-1}}
=\Delta_{{\epsilon},{\kappa}}\parens{\frac{\pi^{{\epsilon +
\kappa}}K_{} - K_{}^{-1}}{\pi q - q^{-1}}}.
\end{align*}
The lemma is proved.
\end{proof}

\begin{lemma}\label{lem:coprod2.2}
The maps $\Delta_{\epsilon,\kappa}$ are coassociative, that is, for
$\epsilon, \kappa, \iota\in \set{0,1}$, the following diagram is
commutative:
\begin{center}
\begin{tikzpicture}[
cross line/.style={preaction={draw=white, -,
line width=6pt}}]
\matrix (m)
[matrix of math nodes, row sep=3em, column sep=4em, ampersand replacement=\&]
{ U_{\epsilon+\kappa+\iota}\&  U_{\epsilon}\otimes U_{\kappa+\iota}  \\
U_{\epsilon+\kappa}\otimes U_{\iota} \& U_{\epsilon}\otimes U_{\kappa} \otimes U_{\iota} \&  \\};
\draw[->] (m-1-1) -- (m-1-2) node[midway, above] {$\Delta_{\epsilon,\kappa+\iota}$};
\draw[->] (m-1-1) -- (m-2-1) node[midway, left] {$\Delta_{\epsilon+\kappa,\iota}$};
\draw[->] (m-2-1) -- (m-2-2) node[midway, below] {$\Delta_{\epsilon,\kappa}\otimes \id$};
\draw[->] (m-1-2) -- (m-2-2) node[midway, right] {$\id\otimes \Delta_{\kappa,\iota}$};
\end{tikzpicture}
\end{center}
\end{lemma}
\begin{proof}
We shall suppress subscripts on elements in
$U_{\epsilon+\kappa+\iota}$. It suffices to check the commutativity
on the generators; it is trivially true on $K$. We compute
\begin{align*}
(\id\otimes \Delta_{\kappa,\iota}) \circ
\Delta_{\epsilon,\kappa+\iota}(F)
 &=F_\epsilon\otimes K_\kappa^{-1}
\otimes K_\iota^{-1}+ e_\epsilon\otimes F_\kappa\otimes
K_\iota^{-1}+e_\epsilon\otimes e_\kappa \otimes F_\iota,
 \\
(\Delta_{\epsilon,\kappa}\otimes \id) \circ
\Delta_{\ep+\kappa,\iota}(F)
 &=F_\epsilon\otimes K_\kappa^{-1}
\otimes K_\iota^{-1}+ e_\epsilon\otimes F_\kappa\otimes
K_\iota^{-1}+e_\epsilon\otimes e_\kappa \otimes F_\iota,
 \\
(\id\otimes \Delta_{\kappa,\iota}) \circ
\Delta_{\epsilon,\kappa+\iota}(E)
 &=E_\epsilon \otimes e_\kappa
\otimes e_\iota+ \pi^{\epsilon} K_\epsilon\otimes E_\kappa\otimes
e_\iota+\pi^{\epsilon+\kappa} K_\epsilon\otimes K_\kappa \otimes
E_\iota,
 \\
(\Delta_{\epsilon,\kappa}\otimes\id) \circ \Delta_{\epsilon +
\kappa,\iota}(E)
 &=E_\epsilon \otimes e_\kappa \otimes e_\iota+
\pi^{\epsilon} K_\epsilon\otimes E_\kappa\otimes
e_\iota+\pi^{\epsilon+\kappa} K_\epsilon\otimes K_\kappa \otimes
E_\iota.
\end{align*}
The lemma is proved.
\end{proof}

\begin{prop}
The superalgebra $\cU$ endowed with the additional structures  below
is a Hopf superalgebra:
\begin{enumerate}
\item
a coproduct $\Delta:\cU\rightarrow \cU\otimes \cU$ defined by
$\Delta=(\Delta_{0,0}+\Delta_{1,1})\oplus
(\Delta_{0,1}+\Delta_{1,0})$;

\item
a counit $\varepsilon:\cU\rightarrow \Q(q)$ defined by
$\varepsilon(e_1)=\varepsilon(E_0)=\varepsilon(F_0)=0$ and
$\varepsilon(K_0)=1$;

\item
an antipode $S:\cU\rightarrow \cU$ defined by
$S(K_\epsilon)=K_\epsilon^{-1}$, $S(F_\epsilon)=-F_\epsilon
K_\epsilon$ and $S(E_\epsilon)=-\pi^\epsilon K_\epsilon^{-1}
E_\epsilon$, for $\ep=0,1$.
\end{enumerate}
\end{prop}

\begin{proof}
The statements on properties of $\Delta$ are simply a reformulation
of Lemmas~\ref{lem:coprods} and \ref{lem:coprod2.2}. It is trivial
to verify that the counit is indeed an algebra homomorphism and
satisfies the defining commutative diagram for a counit; for
example, to check that $(\varepsilon\otimes 1) \circ
\Delta(E_1)=1\otimes E_1$, we compute
\[
(\varepsilon\otimes 1) \circ \Delta(E_1) = \varepsilon(E_0)\otimes
e_1 + \varepsilon(E_1)\otimes e_0 + \varepsilon(K_0)\otimes E_1 +
\pi \varepsilon(K_1)\otimes e_0 = 1\otimes E_1.
\]

To show that the antipode is an anti-automorphism, it is trivial to
check all except for the commutator relation between $E_\ep$ and
$F_\ep$, which we compute directly:
\begin{align*}
S(E_\epsilon F_\epsilon-\pi F_\epsilon E_\epsilon)
&=\pi S(F_\epsilon)S(E_\epsilon)-\pi^2 S(E_\epsilon)S(F_\epsilon)
 \\
&=\pi \pi^\epsilon F_\epsilon E_\epsilon - \pi^\epsilon K_\epsilon
E_\epsilon F_\epsilon K_\epsilon^{-1}
 =-\pi^\epsilon (E_\epsilon F_\epsilon -\pi F_\epsilon E_\epsilon)
  \\
& =\frac{\pi^\epsilon K^{-1} - K}{\pi q -
q^{-1}}=S\parens{\frac{\pi^\epsilon K_\epsilon - K_\epsilon^{-1}
}{\pi q-q^{-1}}}.
\end{align*}

Then we need to check that $m\circ (S\otimes 1) \circ \Delta=m\circ
(1\otimes S) \circ \Delta = \iota\circ \varepsilon$ on the
generators, where  $\iota:\Q(q)\rightarrow \cU$ is the
$\Q(q)$-linear embedding sending $1\mapsto 1$. This is trivial to
check on $E_1$, $F_1$ and $K_1$ since $U_0\otimes U_1\oplus
U_1\otimes U_0$ is in the kernel of $m$. Checking this equality on
$E_0$, $F_0$, and $K_0$ is essentially the same as the
$U_q(\fsl(2))$-argument; for example,
\[
m\circ (S\otimes 1) \circ \Delta(K_0)=S(K_0)K_0+ S(K_1)K_1=e_0+e_1=1=\varepsilon(K_0).
\]
The proposition is proved.
\end{proof}

The following is a super analogue of \cite[3.1.5]{Lu}.
\begin{lemma}\label{lem:divpowcop}
The following formulas hold for $\Delta:U\rightarrow U\otimes U$ and
$\ep =0,1$:
\begin{align*}
\Delta(E_0^{(p)})&=\sum_{a+b=p} q^{ab} E_0^{(a)}K_0^{b} \otimes
E_0^{(b)} + \sum_{a+b=p} \pi^{b} q^{ab} E_1^{(a)}K_1^{b} \otimes
E_{1}^{(b)},
 \\
\Delta(E_1^{(p)})&=\sum_{a+b=p} q^{ab} E_0^{(a)}K_0^{b} \otimes
E_1^{(b)} + \sum_{a+b=p} \pi^{b} q^{ab} E_1^{(a)}K_1^{b}\otimes
E_0^{(b)},
 \\
\Delta(F_0^{(p)})&=\sum_{a+b=p} (\pi q)^{-ab} F_0^{(a)}\otimes
K_0^{-a}F_0^{(b)} + \sum_{a+b=p} (\pi q)^{-ab} F_1^{(a)}\otimes
K_1^{-a}F_{1}^{(b)},
 \\
\Delta(F_1^{(p)})&=\sum_{a+b=p} (\pi q)^{-ab} F_0^{(a)}\otimes
K_1^{-a}F_1^{(b)} + \sum_{a+b=p} (\pi q)^{-ab} F_1^{(a)}\otimes
K_0^{-a}F_0^{(b)}.
\end{align*}
\end{lemma}
\begin{proof}
The proof of all the four identities are similar, and we will only
give the detail on the first one. To prove the first identity, it is
equivalent to prove that
\begin{align*}
\Delta_{0,0}(E_0^{(p)}) &=\sum_{a+b=p} q^{ab} E_0^{(a)}K_0^{b}
\otimes E_0^{(b)},
  \\
\Delta_{1,1}(E_0^{(p)}) &=\sum_{a+b=p} \pi^{b} q^{ab}
E_1^{(a)}K_1^{b} \otimes E_{1}^{(b)}.
\end{align*}
Let us verify only the formula for $\Delta_{1,1}(E_0^{(p)})$ by
induction on $p$, as the other formula can be similarly verified.
The case for $p=1$ follows directly from Lemma~\ref{lem:coprods}.
Assume now the formula for $\Delta_{1,1}(E_0^{(p)})$ is valid for
some $p$. Then,
\begin{align*}
&\Delta_{1,1}(E_0^{(p)}E_0)
 \\
&= \big( \sum_{a+b=p} \pi^{b} q^{ab} E_1^{(a)}K_1^{b} \otimes
E_{1}^{(b)} \big) \cdot (E_1\otimes e_1 + \pi K_1\otimes E_1)
  \\
 &=\sum_{a+b=p} q^{(a+2)b} [a+1] E_1^{(a+1)}K_1^{b} \otimes
E_{1}^{(b)} + \sum_{a+b=p} \pi^{b+1} q^{ab}[b+1] E_1^{(a)}K_1^{b+1}
\otimes E_{1}^{(b+1)}
  \\
 &\stackrel{(\star)}{=} [p+1]E_1^{(p+1)}\otimes e_1 +\pi^{p+1} [p+1] K_1^{p+1}
\otimes E_{1}^{(p+1)}
  \\
&\qquad + \sum_{a+b=p, a\ge 1, b\ge 1} \big(q^{(a+1)(b+1)} [a]
+\pi^{b+1} q^{ab}[b+1] \big) E_1^{(a)}K_1^{b+1} \otimes
E_{1}^{(b+1)}
  \\
&= [p+1]\sum_{a+b=p+1} \pi^{b} q^{ab} E_1^{(a)}K_1^{b} \otimes
E_{1}^{(b)}.
\end{align*}
The identity ($\star$) above is obtained by shifting $a$ to $a-1$
and $b$ to $b+1$ in the first $\sum$ on the left-hand side. This
completes the proof.
\end{proof}

\subsection{Tensor of Modules}

Let $M$ and $N$ be $U$-modules. Then $M\otimes N$ is a $U\otimes
U$-module via the action
\[(u\otimes v)(m\otimes n)=\pi^{p(v)p(m)} (um)\otimes (vn)\]
for $\Z_2$-homogeneous $v\in U$ and $m\in M$. Composition with the
coproduct $\Delta$ defines a $U$-module structure on $M\otimes N$.

\begin{example}
Consider the tensor module $M=L(1,+)\otimes L(2,+)$, for which we
need only consider the action of $U_1$ under the coproduct
$\Delta_{1,0}$.
Let $w$ be a highest weight vector of $L(1,+)$ and $v$ be a highest
weight vector of $L(2,+)$. Then $M\cong L(3,+)\oplus L(1,+)$.
Indeed, the vector \[F_1v\otimes w-\pi q^{-1}[2]^{-1}v\otimes F_0w\]
is a singular vector generating a copy of $L(1,+)$
 since
\[\Delta_{1,0}(E_1) (F_1v\otimes w)
 =E_1F_1v\otimes w+\pi^{p(E_1)p(F_1v)}(\pi K_1)F_1v\otimes E_1w=v\otimes w,\]
\[\Delta_{1,0}(E_1) (v\otimes F_0w)
 =E_1v\otimes w+\pi^{p(E_1)p(v)}(\pi K_1)v\otimes E_0F_0w=\pi q [2] v\otimes w.\]
\end{example}
\section{Quasi-$R$-matrix of $U$}
\label{sec:CBRmatrix}

\subsection{Quasi-$R$ matrix}

We can define the quasi-$R$-matrix $\Theta$ in our setting (cf.
\cite[Chapter~ 4]{Lu} or \cite[Chapter 7]{Jan} for $U_q(\fsl(2))$).
Set
\begin{equation}  \label{eq:an}
a_n =(-1)^n[n]!(\pi q)^{-\binom{n}{2}}(\pi q-q^{-1})^n\in\cA, \quad \text{
for } n\ge 0.
\end{equation}
(Compare the definition of $a_n$ with \eqref{eq:bform2}.) Let
$\epsilon_1,\epsilon_2\in\set{0,1}$. We formally set
$$
\Theta_{\epsilon_1,\epsilon_2}=\sum_{n\geq 0}
\Theta^n_{\epsilon_1,\epsilon_2}, \qquad \text{ with }
\Theta^n_{\epsilon_1,\epsilon_2}= a_n F_{\epsilon_1}^{(n)}\otimes
E_{\epsilon_2}^{(n)},
$$
where $E_{\epsilon}^{(0)}=F_{\epsilon}^{(0)}=e_\epsilon$, the
idempotent corresponding to $U_\epsilon$. Then
$\Theta_{\epsilon_1,\epsilon_2}$ lies in some completion of
$U_{\epsilon_1}\otimes U_{\epsilon_2}$, and it can be regarded as a
well-defined linear operator on the tensor product of
finite-dimensional weight $\cU$-modules. Below we denote
$\overline{u_1\otimes u_2}=\bar{u_1} \otimes \bar{u_2}$ for
$u_1,u_2\in U$ and set
$\bar{\Delta}=\barmap\circ\Delta\circ\barmap$.

\begin{prop}\label{prop:thetacomm}
Let $\ep_1,\ep_2\in \{0,1\}$, and let $u\in U_{\epsilon_1+\epsilon_2}$. Then
\begin{enumerate}
\item
$\Delta_{\epsilon_1,\epsilon_2}(u)\Theta_{\epsilon_1,\epsilon_2}
=\Theta_{\epsilon_1,\epsilon_2}\bar{\Delta}_{\epsilon_1,\epsilon_2} (u)$;

\item
 $\Theta_{\epsilon_1,\epsilon_2}\bar{\Theta}_{\epsilon_1,\epsilon_2}
 =e_{\epsilon_1}\otimes e_{\epsilon_2}$.
\end{enumerate}
\end{prop}

\begin{proof}
To avoid cumbersome notation, we will drop the subscripts on
$E,F,K$; the hidden subscripts can be recovered from the positions
in the tensors.

(1) If
$\Delta_{\epsilon_1,\epsilon_2}(u_1)\Theta_{\epsilon_1,\epsilon_2}
=\Theta_{\epsilon_1,\epsilon_2}\bar{\Delta}_{\epsilon_1,\epsilon_2}
(u_1)$ and
$\Delta_{\epsilon_1,\epsilon_2}(u_2)\Theta_{\epsilon_1,\epsilon_2}
=\Theta_{\epsilon_1,\epsilon_2}\bar{\Delta}_{\epsilon_1,\epsilon_2}
(u_2)$, then clearly
$\Delta_{\epsilon_1,\epsilon_2}(u_1u_2)\Theta_{\epsilon_1,\epsilon_2}
=\Theta_{\epsilon_1,\epsilon_2}\bar{\Delta}_{\epsilon_1,\epsilon_2}
(u_1u_2)$. Hence it suffices to check (1) on the generators $E, F,
K$, which is equivalent to proving the following identities:
\begin{enumerate}
\item[(i)]
$(E\otimes e)\Theta^{n}_{\epsilon_1,\epsilon_2}+(\pi^{\epsilon_1}
K\otimes E) \Theta^{n-1}_{\epsilon_1,\epsilon_2}
=\Theta^{n}_{\epsilon_1,\epsilon_2}(E\otimes e) +
\Theta^{n-1}_{\epsilon_1,\epsilon_2} (K^{-1}\otimes E)$;

\item[(ii)]
$(e\otimes F) \Theta^{n}_{\epsilon_1,\epsilon_2} +(F\otimes
K^{-1})\Theta^{n-1}_{\epsilon_1,\epsilon_2}
=\Theta^{n}_{\epsilon_1,\epsilon_2}(e\otimes F)
+\Theta^{n-1}_{\epsilon_1,\epsilon_2}(F\otimes \pi^{\epsilon_2} K)$;

\item[(iii)]
$(K\otimes K)\Theta^{n}_{\epsilon_1,\epsilon_2}
=\Theta^{n}_{\epsilon_1,\epsilon_2} (K\otimes K)$.
\end{enumerate}

For (i), we have
\begin{align*}
(E\otimes e)\Theta^{n}_{\epsilon_1,\epsilon_2}
&-\Theta^{n}_{\epsilon_1,\epsilon_2}(E\otimes e)
=a_n(E_{}F_{}^{(n)}-\pi^n F_{}^{(n)}E_{})\otimes E_{}^{(n)}
 \\
&=\pi^{1-n} F_{}^{(n-1)}a_n\parens{\frac{(\pi
q)^{1-n}\pi^{\epsilon_1} K_{} - q^{n-1}K_{}^{-1}}{\pi q -
q^{-1}}}\otimes E_{}^{(n)}
 \\
&=\frac{\pi^{1-n}
a_n}{a_{n-1}[n]}\Theta^{n-1}_{\epsilon_1,\epsilon_2}
\parens{\frac{(\pi q)^{1-n}\pi^{\epsilon_1} K_{}
- q^{n-1}K_{}^{-1}}{\pi q - q^{-1}}}\otimes E_{},
\end{align*}
and
\begin{align*}
(\pi^\ep K_{}\otimes E_{})
\Theta^{n-1}_{\epsilon_1,\epsilon_2}&-\Theta^{n-1}_{\epsilon_1,\epsilon_2}
(K_{}^{-1}\otimes E_{})
 \\
&=q^{1-n}(\pi q - q^{-1})
\Theta^{n-1}_{\epsilon_1,\epsilon_2}\parens{\frac{(\pi q)^{1-n}
\pi^\ep K_{} - q^{n-1}K_{}^{-1}}{\pi q - q^{-1}}}\otimes E_{}.
\end{align*}
Hence (i) follows by applying \eqref{eq:an}.

For (ii), we have
\begin{align*}
(e_{}\otimes
F_{})\Theta^{n}_{\epsilon_1,\epsilon_2}&-\Theta^{n}_{\epsilon_1,\epsilon_2}(e_{}
\otimes F_{}) =a_n F_{}^{(n)}\otimes (\pi^n F_{}E_{}^{(n)}-
E_{}^{(n)}F_{})
  \\
&=a_n F_{}^{(n)}\otimes \parens{\pi^{1-n} E_{}^{(n-1)} \frac{q^{1-n}
K_{}^{-1}-(\pi q)^{n-1}\pi^{\epsilon_2} K_{}}{\pi q - q^{-1}}}
 \\
&=\frac{ a_n}{a_{n-1}[n]}\Theta^{n-1}_{\epsilon_1,\epsilon_2}
F_{}\otimes \parens{\frac{q^{1-n} K_{}^{-1}-(\pi
q)^{n-1}\pi^{\epsilon_2} K_{}}{\pi q - q^{-1}}},
\end{align*}
and
\begin{align*}
(F_{}\otimes K_{}^{-1})\Theta^{n-1}_{\epsilon_1,\epsilon_2}
&-\Theta^{n-1}_{\epsilon_1,\epsilon_2}(F_{}\otimes \pi^{\epsilon_2}
K_{})
 \\
&=\pi^{1-n} q^{1-n}(\pi q - q^{-1})
\Theta^{n-1}_{\epsilon_1,\epsilon_2}F_{}\otimes \parens{
\frac{q^{1-n} K_{}^{-1}-(\pi q)^{n-1} \pi^{\epsilon_2} K_{} }{ \pi q
- q^{-1} } }.
\end{align*}
Hence (ii) follows. The identity (iii) is clear.

(2) Write the formal product
\[
\Theta_{\epsilon_1,\epsilon_2}\bar\Theta_{\epsilon_1,\epsilon_2}=\sum_{n\geq
0} b_n F_{}^{(n)}\otimes E_{}^{(n)}.
\]
Comparing coefficients, we compute that $b_0=1$, and for $n \ge 1$,
\[
b_n=[n]!(\pi q - q^{-1})^n \sum_{t=0}^n
(-1)^t\pi^{n(n-t)}(q^{-1})^{-\binom{t}{2}}(\pi
q)^{-\binom{n-t}{2}}\bbinom{n}{t}=0,
\]
where the last equality follows from a version of $q$-binomial
identity for super binomial coefficients. Hence
$\Theta_{\epsilon_1,\epsilon_2}\bar \Theta_{\epsilon_1,\epsilon_2}
=e_{\epsilon_1}\otimes e_{\epsilon_2}$.
\end{proof}

Set $\Theta=\Theta_{0,0}+\Theta_{0,1}+\Theta_{1,0}+\Theta_{1,1}$.
\begin{cor}
We have $\Delta(u)\Theta=\Theta\bar{\Delta}(u)$, for $u\in U$, and
$\Theta\bar{\Theta}=1\otimes 1$.
\end{cor}

Define an antilinear operator
$$
\Psi=\Theta\circ (\bar{\phantom{c}}\times\bar{\phantom{c}})
$$
on $M_1\otimes M_2$ as in \cite[24.3.2]{Lu}, where $M_1$ and $M_2$
are finite-dimensional weight $\cU$-modules. The following can be
proved as in {\em loc. cit.}.

\begin{prop}
The operator $\Psi$ acts as an antilinear involution on the
$\Q(q)$-vector space $M_1 \otimes M_2$, where $M_1$ and $M_2$ are
finite-dimensional $\cU$-modules.
\end{prop}

\subsection{Canonical basis for $^{\omega}L(s)\otimes L(t)$}
 \label{subsec:tensorCB}

Suppose $M$ is a $\cU$-module. We define $^\omega M$ to be the same
vector space as $M$ but with the $\cU$-module action given by
$u\cdot m = \omega(u)m$. In particular, a highest weight module
becomes a lowest weight module under this transformation. Given
$n\in\Z$, we define
$$p(n) \in \{0,1\}  \text{ such that } \parity n\equiv n \text{ (mod 2)}.
$$
Consider the $\cU$-module
\[
L(s,t)=\vphantom{|}^{\omega}L(s)\otimes L(t), \qquad \text{ for }
s,t\in \N.
\]
This module has a basis
\[
E_{\parity s}^{(a)}\eta\otimes F_{\parity{t}}^{(b)}\nu, \qquad 0\leq
a \leq s,\, 0\leq b \leq t,
\]
where  $\eta,\nu$ are the lowest weight and highest weight vectors
respectively. This basis also generates a $\cA$-submodule
$_\cA\cL(s,t)$ which is also an $_\cA \cU$-module. Note that
$\Theta$ and $\Psi$ are well defined on $L(s,t)$ and $_\cA\cL(s,t)$.

Now we have $\Psi(E_{\parity s}^{(a)}\eta\otimes F_{\parity
t}^{(b)}\nu)=E_{\parity s}^{(a)}\eta\otimes F_{\parity t}^{(b)}\nu +
(*)$, where  ($*$) is an $\cA$-linear combination of $E_{\parity
s}^{(i)}\eta\otimes F_{\parity t}^{(j)}\nu$, with $(i,j)\prec
(a,b)$. Here the partial order $\preceq$ on $\N^2$ is defined by
declaring that $(i,j)\preceq (m,n)$ if and only if $m-n=i-j$ and
$m\leq i$ (hence also $n\leq j$). Then by a variant of \cite[Lemma
24.2.1]{Lu} adapted to our bar map \eqref{eq:bar}, we have the
following.

\begin{prop}  \label{prop:CBtensor}
Retain the notations above. There exists a unique $\Psi$-invariant
element $(E^{(a)}\diamondsuit F^{(b)})_{s,t} \in {}_\cA\cL(s,t)$,
for $0\leq a \leq s,\, 0\leq b \leq t,$ such that
$$
(E^{(a)}\diamondsuit F^{(b)})_{s,t}=\sum_{m,n}c_{a,b;m,n}^{s,t}
E_{\parity s}^{(m)}\eta \otimes F_{\parity t}^{(n)} \nu,
$$
where $c_{a,b;a,b}^{s,t}=1$, $c_{a,b;m,n}^{s,t}\in
q^{-1}\Z[q^{-1}]$, for all $(m,n)\prec (a,b)$.
\end{prop}

This is an analogue of \cite[Theorem 24.3.3]{Lu}. The elements
$(E^{(a)}\diamondsuit F^{(b)})_{s,t}$, for $0\leq a \leq s,\, 0\leq
b \leq t,$ will be called the {\em canonical basis} of $L(s,t)$. The
coefficients $c_{a,b;m,n}^{s,t}$ will be determined precisely in
Corollary~\ref{cor:const}.

\section{Modified superalgebra and canonical basis}
\label{sec:CBdotU}

\subsection{Algebra $\dotU$}
\label{sec:defdotU}

Let $a,b\in \Z$, and consider the subspace  of $\cU$:
\[
{}_aJ_b=(K_{\parity
a}-q^{a}1)U_{\parity{a}}+U_{\parity{a}}(K_{\parity b}-q^{b}1).
\]
Then ${}_{a}J_b$ is a subspace of $U_{\parity{a}}$, and
${}_{a}J_b=U_{\parity{a}}$ if $\parity a\neq \parity b$. We set
\[
{}_a \cU {}_b=U_{\parity{a}}/{}_aJ_b.
\]
Note that ${}_a \cU {}_b=\set{0}$ if $\parity a \neq \parity b$.

We define
$$
\dotU=\bigoplus_{m,n\in \Z}\phantom{|}{}_m\cU_n.
$$
This is called the {\em modified} (also called {\em idempotented})
quantum enveloping algebra of $\osp(1|2)$ (cf. \cite{BLM, Lu}). Let
$p_{m,n}: \cU\rightarrow {}_m\cU_n$ be the canonical projection. We
endow $\dotU$ with the structure of an associative algebra under the
multiplication
\begin{equation} \label{eq:mult}
p_{k,\ell}(x)p_{m,n}(y)=\delta_{\ell,m}p_{k,n}(xy), \quad \text{ for
} x,y\in \cU;\  k,\ell,m,n\in \Z.
\end{equation}

The algebra $\dotU$ inherits a $\Z$-grading from $\cU$:
$$
\dotU = \bigoplus_{k\in 2\Z} \dotU(k),
$$
where
\[
\dotU(k)=\sum_{m,n\in \Z} p_{m,n}(\cU(k)).
\]

Note that if $x\in \cU(2i)$, then $p_{m,n}(x)=0$ if $2i\neq m-n$,
since the identity $q^{2i}x=KxK^{-1}$ in $\cU$ descends to
$q^{2i}p_{m,n}(x)=q^{m-n}p_{m,n}(x)$.  The new feature in this
algebra is the addition of idempotents $1_n=p_{n,n}(1)$, which
satisfy
$$
1_m 1_n=\delta_{m,n} 1_n.
$$
We have
$$
\phantom{|}{}_m\cU_n=1_m\dotU 1_n.
$$
Also, we have that $\dotU=\dot{U_0}\oplus \dot{U_1}$, where
$$
\dot{U_\epsilon}=\sum_{a,b\in \Z} 1_{2a+\epsilon}\dotU
1_{2b+\epsilon}.
$$
Moreover, $\dotU_0$ and $\dotU_1$ are subalgebras of $\dotU$ such
that $\dot{U_0}\dot{U_1}=\dot{U_1}\dot{U_0}=0$.

\subsection{$\dotU$ as a $\cU$-bimodule}
\label{sec:Ubimod}

The algebra $\dotU$ has a natural $\cU$-bimodule structure: if $x\in
\cU(k)$, $y \in\dotU$ and $z\in \cU(n)$ then we set
\begin{equation} \label{eq:bim}
xp_{\ell,m}(y)z=p_{k+\ell,m-n}(xyz).
\end{equation}
With this action, we have the following identities in $\dotU$, for
$n\in\Z, a\in\N,\ep=0,1$:
\begin{align}\label{idempmult}
 \begin{split}
E_\epsilon^{(a)}1_n=
&\delta_{\epsilon,\parity{n}}1_{n+2a}E_\epsilon^{(a)},\qquad
F_\epsilon^{(a)}1_n=\delta_{\epsilon,\parity{n}}1_{n-2a}F_\epsilon^{(a)},
  \\
&(E_\epsilon F_\epsilon-\pi F_\epsilon E_\epsilon) 1_n
=\delta_{\epsilon,\parity n}[n]1_n,
 \end{split}
\end{align}
\begin{equation}\label{Kmult}
[K_\epsilon;m]1_n=\delta_{\epsilon,\parity n}[n+m]1_n, \qquad
\bbinom{K_\epsilon;m}{a}1_n=\delta_{\epsilon,\parity n}\bbinom{m+n}{a} 1_n.
\end{equation}

The following is a super analogue of \cite[23.1.3]{Lu}.

\begin{prop}
The following identities hold in $\dotU$: for $n\in \Z$, $r,s \ge
0$,
\begin{align}
\pi^{rs} E_\epsilon^{(r)}1_nF_\epsilon^{(s)}
&=\delta_{\epsilon,\parity n}\sum_{i=0}^{\min(r,s)}
\pi^{\binom{i+1}{2}} \bbinom{n+(r+s)}{i}
F_\epsilon^{(s-i)}1_{n+2s+2r-2i}E_\epsilon^{(r-i)},
  \label{dotdivpowcom1} \\
\pi^{rs} F_\epsilon^{(s)}1_nE_\epsilon^{(r)}
&=\delta_{\epsilon,\parity n}\sum_{i=0}^{\min(r,s)}
\pi^{\binom{i}{2}+\epsilon i}\bbinom{(r+s)-n}{i}
E_\epsilon^{(r-i)}1_{n-2s-2r+2i}F_\epsilon^{(s-i)}.
  \label{dotdivpowcom2}
\end{align}
\end{prop}

\begin{proof}
First, it is clear by definition that the expressions are zero
unless the parities match, so we may assume that $\epsilon=\parity
n$. Using \eqref{idempmult}, \eqref{Kmult} and Lemma~
\ref{divpowcom}, we compute that
\begin{align*}
\pi^{rs} E_\epsilon^{(r)}1_nF_\epsilon^{(s)} &=\parens{\sum_{i=0}^{\min(r,s)}
\pi^{\binom{i+1}{2}}
F_\epsilon^{(s-i)}\bbinom{K_\epsilon;2i-(r+s)}{i}E_\epsilon^{(r-i)}}1_{n+2s}
 \\
&=\sum_{i=0}^{\min(r,s)} \pi^{\binom{i+1}{2}} F_\epsilon^{(s-i)}
\bbinom{K_\epsilon;2i-(r+s)}{i}1_{n+2s+2r-2i}E_\epsilon^{(r-i)}
 \\
&=\sum_{i=0}^{\min(r,s)} \pi^{\binom{i+1}{2}} \bbinom{n+(r+s)}{i}
F_\epsilon^{(s-i)}1_{n+2s+2r-2i}E_\epsilon^{(r-i)}.
\end{align*}
This proves \eqref{dotdivpowcom1}. The identity
\eqref{dotdivpowcom2} can be proved similarly, using in addition the
identities \eqref{eq:n-n}.
\end{proof}

\subsection{Additional structures of $\dotU$}

We also note that $\dotU$ has a triangular decomposition as in
Lustzig \cite[23.2]{Lu}.  Recall the algebra $\f$ from
\S\ref{subsec:f}. The $\cU$-bimodule structure induces a $(\f,
\f^{\rm op})$-bimodule structure on $\dotU$ via
$$
x\otimes y \cdot u=x^{-}u y^+, \quad \text{ for } x,y\in \f, u\in
\dotU.
$$

Recall that
$F_\epsilon^{(a)}1_nE_\epsilon^{(b)}=0=E_\epsilon^{(b)}1_n
F_\epsilon^{(a)}$ if and only if $\ep \neq p(n)$. Hence we adopt the
following convention by dropping the subscript $\ep$ without
ambiguity:
\begin{equation}  \label{eq:conv}
F^{(a)}1_nE^{(b)}:=F_{p(n)}^{(a)}1_nE_{p(n)}^{(b)}, \quad
E^{(a)}1_nF^{(b)}:=E_{p(n)}^{(a)}1_nF_{p(n)}^{(b)}.
\end{equation}
In this way, we could also drop all subscripts $\ep$ as well as
$\delta_{\epsilon,\parity n}$ in
\eqref{idempmult}-\eqref{dotdivpowcom2}.

It follows by the triangular decomposition of $\cU$ that the
elements $F^{(a)}1_nE^{(b)}$, for  $n\in \Z, a,b\in \N$, form a
basis for $\dot{U}$. Similarly, $E^{(b)}1_n F^{(a)}$, for $n\in \Z,
a,b\in \N$ form a basis for $\dotU$. In addition, it is clear from
\eqref{dotdivpowcom1} and \eqref{dotdivpowcom2} that these two bases
span the same $\cA$-submodule of $\dotU$, denoted by ${}_\cA \dotU
$. This $\cA$-submodule $\dotAU $ is in fact an $\cA$-subalgebra
generated by the elements $E^{(a)}1_n$ and $F^{(a)}1_n$, for $n\in
\Z, a\in\N$.

We say a $\dotU$-module is {\em unital} if for every $v\in M$, $1_n
v=0$ for all but finitely many $n\in \Z$ and $v=\sum_{n\in \Z} 1_n
v$. Each unital module is a weight $\cU$-module under the action
$u\cdot v = \sum_{n\in \Z} (u1_n)v$, where $u1_n$ is viewed as an
element of $\dotU$. Likewise, each weight $\QU$-module with weights
in $q^\Z$ is naturally a unital $\dotU$-module: given a weight
decomposition $v=\sum_{n\in \Z} v_n$ such that $Kv_n =q^n v_n$, we
set $1_n v = v_n$.

We define $\Delta_{a,b,c,d}: {}_{a+b}\cU_{c+d}\rightarrow
{}_{a}\cU_{c}\otimes {}_{b}\cU_{d}$ by (cf. \cite[23.1.5]{Lu})
\[
\Delta_{a,b,c,d}(p_{a+b, c+d}(x))= (p_{a,c}\otimes p_{b,d}) \circ
\Delta(x).
\]
The direct product of these maps for various $a,b,c,d$ defines a
coproduct on $\dotU$ which restricts to $\cA$-linear homomorphism on
${}_\cA\dotU$.

The antilinear bar-involution $\ \bar{\phantom{c}} : \cU \rightarrow
\cU$ induces an antilinear bar-involution $\ \bar{\phantom{c}} :
\dotU \rightarrow \dotU$, which fixes each idempotent $1_n$ for
$n\in \Z$, and satisfies $\overline{xhy}=\bar{x}\bar{h}\bar{y}$ for
$x,y\in U$ and $h \in\dotU$. Similarly, the (anti-)automorphisms
$\omega$, $\tau$ and $\rho$ on $U$ induce (anti-)automorphisms on
$\dotU$ (denoted by the same letters), which respect the
$\cU$-bimodule structure, and $\rho(1_n)=1_{n}$,
$\omega(1_n)=1_{-n}$, $\tau(1_n)=1_{-n}$, for $n\in\Z$.

\subsection{Canonical basis for $\dotU$}

Following Lusztig \cite{Lu}, a canonical basis for $\dotU$ should be
a bar-invariant $\Q(q)$-basis for $\dotU$ and an
$\cA$-basis for ${}_\cA \dotU$ which consist of elements of the form
$$
u=E^{(a)}\diamondsuit_k F^{(b)}\in {}_\cA \dotU 1_k, \text{ for }
a,b\in\N, k\in \Z,
$$
such that $u(\eta\otimes\nu)=(E^{(a)}\diamondsuit
F^{(b)})_{s,t}$ where $\eta$ is the lowest weight vector for
$^\omega L(s)$ and $\nu$ is the highest weight vector for $L(t)$,
with $t-s=k$. We take this as the definition of a canonical basis
for $\dotU$.

Keeping in mind the convention \eqref{eq:conv}, we consider the
elements
\begin{equation}  \label{cbelem}
E^{(a)}1_{-n}F^{(b)}, \quad \pi^{ab} F^{(b)}1_{n}E^{(a)}, \quad \text{ for }
a,b\in\N, n\in \Z,  n\geq a+b.
\end{equation}
By \eqref{dotdivpowcom1}, we have the following overlapping elements
in \eqref{cbelem}:
\begin{equation}  \label{cbelem=}
E^{(a)}1_{-n}F^{(b)}=\pi^{ab}F^{(b)}1_{n}E^{(a)}, \quad \text{ for } n=a+b.
\end{equation}

The following is a super analogue of \cite[Proposition~25.3.2]{Lu},
and it formally looks identical!

\begin{theorem}  \label{th:CBdot}
The elements in \eqref{cbelem} subject to the identification
\eqref{cbelem=} form a canonical basis for $\dotU$. Moreover, if
$n\ge a+b$, we have
\begin{align*}
E^{(a)}1_{-n}F^{(b)} &= E^{(a)}\diamondsuit_{2b-n}F^{(b)},
 \\
\pi^{ab} F^{(b)}1_n E^{(a)} &= E^{(a)}\diamondsuit_{n-2a}F^{(b)}.
\end{align*}
\end{theorem}

\begin{proof}
First, recall that all elements of the form $E^{(a)}1_nF^{(b)}$ form
a basis for the $\cA$-algebra ${}_\cA \dotU$ and $\Q(q)$-algebra
$\dotU$. If $a+b>n$, $E^{(a)}1_{-n}F^{(b)}$ can be expressed as a
$\cA$-linear combination of the elements in \eqref{cbelem} by using
\eqref{dotdivpowcom1} as follows:
\[
\pi^{ab} E^{(a)}1_{-n}F^{(b)}=\sum_{i=0}^{\min(a,b)}
\pi^{\binom{i+1}{2}} \bbinom{a+b-n}{i}
F^{(b-i)}1_{2a+2b-n-2i}E^{(a-i)},
\]
where
$
0\le a-i+b-i<(a+b-n)+a+b-2i=2a+2b-n-2i.
$
Hence we conclude that the set \eqref{cbelem} forms a spanning set
of $\dotU$. On the other hand, the set \eqref{cbelem} naturally
splits into two halves, each of which is already linearly
independent. Except for the case $a+b=n$ with identification
\eqref{cbelem=}, the halves live in different subspaces ${}_a U_b$
and hence are necessarily linearly independent. This shows the
linear independence of the set \eqref{cbelem} subject to the
identification \eqref{cbelem=}.

Let $\eta_s$ and $\nu_t$ be the lowest and highest weight vectors of
$^\omega L(s)$ and $L(t)$. We have
$E^{(a)}1_{-n}F^{(b)}(\eta_s\otimes \nu_t)=0$ unless $-n+2b=t-s$, in
which case we compute by Lemma~\ref{lem:divpowcop} that
\begin{align*}
& E^{(a)}1_{-n}F^{(b)} (\eta_s\otimes \nu_t)
 \\
&=\Delta(E^{(a)})\Delta(F^{(b)})(\eta_s\otimes \nu_t)
=\Delta(E^{(a)})(\eta_s\otimes F^{(b)}\nu_t)
 \\
&=\sum_{a=c+d}\pi^{sd} q^{cd}E^{(c)}K^d \eta_s\otimes
E^{(d)}F^{(b)}\nu_t
 \\
&=\sum_{a=c+d}\pi^{sd}q^{dc-ds}E^{(c)} \eta_s\otimes
E^{(d)}F^{(b)}\nu_t
 \\
&=\sum_{a=c+d}\sum_{i=0}^{\min(b,d)}\pi^{sd}q^{dc-ds}E^{(c)}\eta_s\otimes
 \pi^{-bd}\pi^{\binom{i+1}{2}} F^{(b-i)}\bbinom{K;2i-(b+d)}{i} E^{(d-i)}\nu_t
  \\
&=\sum_{a=c+d}\pi^{sd}q^{dc-ds}E^{(c)} \eta_s\otimes
 \pi^{-bd}\pi^{\binom{d+1}{2}} \bbinom{d-b+t}{d} F^{(b-d)}\nu_t
  \\
&=\sum_{0\leq j\leq \min(a,b)}
\pi^{sj+\binom{j+1}{2}-bj}q^{j(a-j-s)} \bbinom{j-b+t}{j} E^{(a-j)}
\eta_s\otimes  F^{(b-j)}\nu_t.
\end{align*}

Let us denote by $X$ the right-hand side of the last equation. Then
$X$ is bar-invariant since the left-hand side is; it is also
therefore $\Theta$-invariant since $\Theta (\eta_s\otimes \nu_t) =
\eta_s\otimes \nu_t$, so $X$ is $\Psi$-invariant. The leading term
(i.e., the term with $j=0$) of $X$ is $E^{(a)} \eta_s\otimes
F^{(b)}\nu_t$. If $j>0$, a degree argument shows that
$q^{j(a-j-s)}\bbinom{j-b+t}{j}$ lies in $q^{-1}\Z[q^{-1}]$. Hence
$X$ satisfies the defining properties of the element
$(E^{(a)}\diamondsuit F^{(b)})_{s,t}$ (see
Proposition~\ref{prop:CBtensor}), and then must be equal. A similar
argument applies to $F^{(b)}1_{n}E^{(a)}$.

It is clear from the triangular decomposition and the definition of
$\, \bar{\phantom{c}} \;$ that the other properties of a canonical
basis are satisfied, completing the proof.
\end{proof}

From the proof above, we have the following formula for the
coefficients $c_{a,b;m,n}^{s,t}$ in the expansion of
$(E^{(a)}\diamondsuit F^{(b)})_{s,t}$ as defined in
Proposition~\ref{prop:CBtensor}.

\begin{cor}  \label{cor:const}
Let $0\leq a \leq s,\, 0\leq b \leq t$. For $0\le j\leq \min(a,b)$,
we have
$$
c_{a,b;a-j,b-j}^{s,t}= \pi^{sj+\binom{j+1}{2}-bj}q^{j(a-j-s)}
\bbinom{j-b+t}{j}.
$$
\end{cor}

\subsection{A bilinear form on $\dotU$}\label{sec:dotinnerprod}

Recall the definition of $\rho$ from Proposition~\ref{prop:autom}.
Since we have defined a suitable bilinear form $(\cdot,\cdot)$ on
$\f$ (see \eqref{eq:bform1} and \eqref{eq:bform2}) and constructed
the canonical basis on $\dotU$, the same proof in \cite[26.1.2]{Lu}
leads to the following.

\begin{prop}
There exists a unique bilinear form $(\cdot,\cdot):\dotU\times
\dotU\rightarrow \Q(q)$ such that
\begin{enumerate}
\item $(1_a x 1_b, 1_c y1_d)=0$ whenever $a\neq c$ or $b\neq d$, $a,b,c,d\in \Z$;
\item $(ux,y)=(x,\rho(u) y)$ for $u\in \QU$ and $x,y\in\dotU$;
\item $(x^- 1_a, y^- 1_a)=(x,y)$ for all $x,y\in \f$ and $a\in \Z$.
\end{enumerate}
Moreover, the bilinear form $(\cdot,\cdot)$ is symmetric.
\end{prop}
\section{The covering algebras}
\label{sec:covering}

Essentially all the constructions and results in the previous
sections make sense in the framework of covering algebras introduced
below by treating $\pi$ as a formal parameter satisfying $\pi^2=1$.
The idea of (half) covering algebras first appeared in \cite{HW}.
Given a ring $A$ with unit, we define a new ring
$A^\pi=A[\pi]/(\pi^2-1).$ We shall mainly need $\piA$ and $\piQ$
below. Note that $\piA \subset \piQ$.
The quantum integers and quantum binomials $[n],\bbinom{n}{i}$ in
\eqref{eq:qn} and \eqref{eq:qbinom} make sense as elements in $\piA$
and also in $\piQ$.

\subsection{Covering algebra $\piU$}

We define the {\em covering algebra $\piU$} for $\osp(1|2)$ to be
the $\piQ$-(super)algebra generated by elements
$E_\epsilon,F_\epsilon,K_\epsilon$ and $K_\epsilon^{-1}$ for
$\epsilon\in\set{0,1}$, subject to the relations (1)-(3) in
Remark~\ref{rem:epsrels}. Then all the definitions and calculations
earlier on can be translated to the covering algebra. Indeed, all
computations only involve quotients of elements of the form $(\pi
q)^n - q^{-n}$ and we never used $1+\pi=0$ to reduce any expression.
Therefore we have the following.
\begin{enumerate}
\item
$\piU$ is a free $\piQ$-module with basis
$F_\epsilon^{(a)}K_\epsilon^bE_\epsilon^{(c)}$ for $a,c\in \N$,
$b\in \Z$, $\epsilon\in \set{0,1}$.

\item
$\piU$ has algebra (anti-)automorphisms as described in
Proposition~\ref{prop:autom} which fix $\pi$.

\item
The elements $E^{(r)}$, $F^{(s)}$ satisfy the commutation relations
in Lemma~\ref{divpowcom}.


\item
$\piU$ has a Hopf superalgebra structure.

\item
$\piU$ admits a quasi-R matrix $\Theta$ and the map $\Psi$ as
operators on tensor products of modules.

\item
Proposition \ref{prop:CBtensor} remains valid, with
$c_{a,b;m,n}^{s,t}\in q^{-1}\N[q^{-1},\pi]$.
\end{enumerate}

\subsection{Covering algebra $\dotpiU$}

Similarly, we can modify the definition of $\dotU$ in
\S\ref{sec:defdotU} as follows. Let $a,b\in \Z$ and set
\[
{}_a\cU^\pi_b=\piU/\parens{(K_{p(a)}-q^{a})\piU+\piU(K_{p(b)}-q^{b})},
\]
and define
$$
\dotU^\pi=\bigoplus_{a,b\in \Z}\phantom{|}{}_a\piU_b.
$$
This is called the {\em modified} (also called {\em idempotented})
covering quantum (super)algebra of $\QU^\pi$. Imitating the
$\cA$-subalgebra $\dotAU$, we can define the $\piA$-subalgebra
$\dotAU^\pi$. We can now reinterpret earlier results on $\dotU$ in
the setting of covering algebra as follows:
\begin{enumerate}
\item
The identities \eqref{idempmult}, \eqref{Kmult},
\eqref{dotdivpowcom1}, and \eqref{dotdivpowcom2} are valid in
$\dotAU^\pi$.

\item
Theorem~\ref{th:CBdot} on canonical basis is valid for $\dotAU^\pi$.
\end{enumerate}

\subsection{Specializations}

The specialization by setting $\pi$ to be $\pm 1$ in the
constructions and statements for the covering algebras recovers
corresponding results for quantum $\fsl(2)$ and $\osp(1|2)$
simultaneously as follows.
\begin{enumerate}
\item
Specializing $\pi=-1$, we obtain that $\cU^\pi/ \langle \pi+1\rangle
\cong \cU$ and $\dotU^\pi/ \langle \pi+1\rangle \cong \dotU$.

\item
The canonical basis for $\dotU^\pi$ specializes at $\pi=-1$ to that
for $\dotU$.

\item
Specializing $\pi=1$, we obtain that $\cU^\pi/ \langle \pi-1\rangle$
is isomorphic to a direct sum of two copies of the quantum group
$U_q(\mathfrak{sl}(2))$, and $\dotU^\pi/ \langle \pi-1\rangle$ is
isomorphic to the modified algebra $\dot{U}_q(\mathfrak{sl}(2))$ in
\cite{BLM, Lu}.

\item
The canonical basis for $\dotU^\pi$ specializes at $\pi=1$ to that
for the modified quantum $\mathfrak{sl}(2)$ given in
\cite[Proposition~25.3.2]{Lu}.
\end{enumerate}

\begin{remark}
The super sign, being inherent in the structure of superalgebras,
rules out the hope of positivity of the structure constants for
canonical basis of the quantum superalgebra $\dotU$ in the usual
sense. Using \eqref{idempmult}, \eqref{dotdivpowcom1} and
\eqref{dotdivpowcom2}, we can show that the structure coefficients
from multiplying canonical basis  elements in $\dotU^\pi$ lie in
$\N[q, q^{-1},\pi]$. So passing to covering algebras restores the
positivity.
\end{remark}

A categorification of $\dotU^\pi$ and its canonical basis, \`a la
Lauda \cite{La} for modified quantum $\mathfrak{sl}(2)$, is expected
in a generalized framework of spin nilHecke algebras, with $\pi$
categorified as a parity shift functor as in \cite{HW}. Such a
categorification would be relevant to odd Khovanov homology and knot
invariants (also compare \cite{B2}). Forgetting the $\Z_2$-grading
and the parity shift functor would lead to a (second)
categorification of modified quantum $\fsl(2)$ and its canonical
basis; see (4) above.

\end{document}